\documentclass[10pt]{article}
\usepackage[usenames]{color}
\usepackage{amssymb}
\usepackage{multicol}
\usepackage{graphicx}
\usepackage{amsthm}
\usepackage{verbatim}
\usepackage{float}
\usepackage{forest,adjustbox}
\usepackage{color}
\usetikzlibrary{arrows,decorations.pathmorphing,backgrounds,positioning,fit,petri,calc}

\usepackage{amsmath, epsfig}

\theoremstyle{plain}
\newtheorem{thm}{Theorem}
\newtheorem{lemma}[thm]{Lemma}

\newtheorem{prop}[thm]{Proposition}

\newtheorem{conj}[thm]{Conjecture}
\newtheorem*{ques*}{Question}

\title{Kings in Multipartite Hypertournaments}
\author{
Jiangdong Ai\thanks{Department of Computer Science. Royal Holloway University of London.  {\tt Jiangdong.Ai.2018@live.rhul.ac.uk}.} \and Stefanie Gerke\thanks{Department of Mathematics. Royal Holloway University of London.  {\tt stefanie.gerke@rhul.ac.uk}.} \and Gregory Gutin \thanks{Department of Computer Science. Royal Holloway University of London. {\tt g.gutin@rhul.ac.uk}.} }
\begin{document}
 \maketitle

 \begin{abstract}
 In his paper ``Kings in Bipartite Hypertournaments'' (Graphs $\&$ Combinatorics 35, 2019), Petrovic stated two conjectures on 4-kings in multipartite hypertournaments. We prove one of these conjectures and give counterexamples for the other.
 \end{abstract}

\section{Introduction}\label{sec:intro}

Given two integers $n$ and $k$, $n\ge k>1$,
a $k$-{\em hypertournament} $T$ on $n$ vertices is a pair $(V,A)$,
where $V$ is a set of vertices, $|V|=n$ and $A$ is a set of $k$-tuples of
vertices, called arcs, so that for any $k$-subset $S$ of $V$,
$A$ contains exactly one of the $k!$ tuples whose entries belong to
$S$. For an arc $x_1x_2\ldots x_k$, we say that $x_i$ {\em precedes} $x_j$ if $i<j.$
A 2-hypertournament is merely an (ordinary) tournament. Hypertournaments have been studied in a large number of papers, see e.g. \cite{A,BB,B,BPT,F,GY,LLGS,PT,Y}.

Recently, Petrovic \cite{P} introduced multipartite hypertournaments in a similar way. Let $n$ and $k$ be integers such that $n> k\ge 2.$ Let $V$ be a set of $n$ vertices and $V=V_1\uplus V_2\uplus \dots \uplus V_p$ be a partition of $V$ into $p\ge 2$ non-empty subsets.
A {\em $p$-partite $k$-hypertournament} (or, {\em multipartite hypertournament}) $H$ can be obtained from a $k$-hypertournament $T$ on vertex set $V$ by deleting all arcs  $x_1x_2\ldots x_k$ such that $\{x_1,x_2,\ldots ,x_k\}\subseteq  V_i$ for some $i\in [p].$
We call $V_i$'s {\em partite sets} of $H.$ The set of arcs of $H=(V,A)$ will be denoted by $A(H),$ i.e., $A(H)=A.$
A $p$-partite 2-hypertournament is a {\em $p$-partite tournament}. 

For $u\in V_i,w\in V_j$ with $i\ne j,$ $A_H(u,w)$ is the set of arcs of $H$ which contain $u$ and $w$ and where $u$ precedes $w.$ We will write $xey$ if $e\in A_H(x,y).$
We let $A_H(x,y) =\emptyset$  if either  $x$ and $y$ belong to the same partite set of $H.$
A {\em path} in $H$ is an alternating sequence $P=x_1a_1x_2a_2\dots x_{q-1}a_{q-1}x_q$ of distinct vertices $x_i$ and distinct arcs $a_j$ such that $x_ja_jx_{j+1}$ for every $j\in [q-1].$
We will call $P$ an $(x_1,x_q)$-{\em path} of {\em length} $q-1.$ 

Let $q\ge 1$ be a natural number. A vertex $x$ of $H$ is a $q$-{\em king} if for every $y\in V$, $H$ has  an $(x,y)$-path of length at most $q.$  Generalizing a well-known theorem of Landau that every tournament has a 2-king (see e.g. \cite{BJG}),
Brcanov et al. \cite{BPT} showed that every hypertournament has a 2-king.
A vertex $v$ of $H$ is a {\em transmitter} if for every vertex $u$ from a different partite set than $v,$ $A_H(u,v)=\emptyset.$

Note that for every $u\in V_i,w\in V_j$ $(i\neq j)$, we have
$|A_H(u,w)|+|A_H(w,u)|=\binom{n-2}{k-2}.$ A {\em majority multipartite tournament} $M(H)$ of $H$ has the same partite sets as $H$ and for every $u\in V_i$ and $w\in V_j$ with $i\ne j$, $uw\in M(H)$ if
$|A_H(u,w)|> \frac{1}{2}\binom{n-2}{k-2}.$ If $|A_H(u,w)| = \frac{1}{2}\binom{n-2}{k-2}$ then we can choose either $uw$ or $wu$ for $M(H).$ 

For a graph $G=(V,E)$ and $U\subseteq V$, let $N_G(U)=\{v\in V\setminus U:\ uv\in E, u\in U\}.$

Gutin \cite{G} and independently Petrovic and Thomassen \cite{PT} proved the following:
\begin{thm}\label{t4-k} \cite{G,PT}
Every multipartite tournament with at most one transmitter contains a 4-king.
\end{thm}
Petrovic \cite{P} proved that the same result holds for bipartite $k$-hypertournaments:

\begin{thm}\label{bt4-k} \cite{P}
Every bipartite $k$-hypertournaments ($k\ge 2$) with at most one transmitter contains a 4-king.
\end{thm}

In the same paper he conjectured the following:

\begin{conj}\label{c1} \cite{P}
Every multipartite $k$-hypertournament  ($k\ge 2$) with at most one transmitter contains a 4-king.
\end{conj}

In this short paper, we will solve this conjecture in the affirmative.

The next conjecture of Petrovic \cite{P} is motivated by the fact that Petrovic and Thomassen \cite{PT} proved that the assertion of the conjecture holds for bipartite tournaments.

\begin{thm}\label{thm:PT}\cite{PT}
Every bipartite tournament $B$  without transmitters has at least two 4-kings in each partite set of $B.$
\end{thm}
\begin{conj}\label{c2} \cite{P}
Every bipartite $k$-hypertournament $B$  ($k\ge 2$) without transmitters has at least two 4-kings in each partite set of $B.$
\end{conj}

In this paper, we will first show a couterexample to Conjecture \ref{c2} and then exhibit a wide family of bipartite hypertournaments for which the conclusion of the conjecture holds. 

The paper is organized as follows. In the next section, we prove a lemma (Lemma \ref{l2}) which we call {\em the Majority Lemma}, and which
is used to show the positive above-mentioned results. In Section \ref{sec:res}, we provide the counterexample and positive results.
The terminology not introduced in this paper can be found in \cite{BJG}.

\section{The Majority Lemma}

The Majority Lemma, Lemma \ref{l2}, is the main technical result of this paper.
To prove Lemma \ref{l2}, we will use the following simple lemma.

\begin{lemma}\label{l1}
Let $G$ be a bipartite graph with partite sets $U$ and $W$ and let every vertex in $U$ have degree at least $p\ge 1$ and every vertex in $W$ have degree at most $p$, except for one vertex which has degree at most $2p-1.$  Then $G$ has a matching saturating $U.$
\end{lemma}
\begin{proof}
By Hall's theorem, if for every $S\subseteq U$, $|S|\le |N_G(S)|$ then $G$ has a matching saturating $U.$
Suppose that there is a subset $S$ of $U$ such that $|S|\ge |N_G(S)|+1.$ Let $e$ be the number of edges in the subgraph of $G$ induced by $S\cup N_G(S)$ and
observe that $$p|S|\le e\le (|N(S)|-1)p +(2p-1)\le  (|S|-2)p +(2p-1)=|S|p-1,$$ a contradiction.
\end{proof}

Proposition \ref{p2} proved in the next section shows that Lemma \ref{l2} cannot be extended to $n=4$ and $p=2.$

\begin{lemma}\label{l2}
Let $H$ be a $p$-partite $k$-hypertournament with $p\ge 2.$
Let $n\ge 5$ and $n>k\ge 3.$ If a majority $p$-partite tournament $M(H)$ has an $(x,y)$-path $P$ of length at most 4, then  $H$ has such a path of length at most 4.
\end{lemma}
\begin{proof}
It suffices to prove this lemma for the case when $P$ is of length 4 as the other cases are simpler and similar. Thus, assume that
 $P=x_1x_2x_3x_4x_5.$ By definition of a path, for every $i\in [4],$ $x_i$ and $x_{i+1}$ belong to different partite sets of $H.$
 Now consider the following cases covering all possibilities.

\vspace{1mm}

 \noindent{\bf Case 1: $n\ge 9$ and $3\le k<n$ or $n\ge 7$ and $4\le k< n-1$.}
 Observe that if for every $i\in \{1,2,3,4\}$, \begin{equation}\label{1} |A_H(x_i,x_{i+1})|>3 \end{equation} then we can choose distinct arcs $a_i\in A_H(x_i,x_{i+1})$ such that
 $x_1a_1x_2a_2x_3a_3x_4a_4x_5$ is the required path in $H$.  In particular, inequalities (\ref{1}) will hold if  $\frac{1}{2}\binom{n-2}{k-2}>3.$

 If $n\ge 9$ and $3\le k<n,$ we have $$\frac{1}{2}\binom{n-2}{k-2}\ge \frac{n-2}{2}>3$$  and hence inequalities (\ref{1}) hold.
 If $n\ge 7$ and $4\le k< n-1,$ we have $$\frac{1}{2}\binom{n-2}{k-2}\ge \frac{(n-2)(n-3)}{4}>3.$$  

 \vspace{1mm}

 \noindent{\bf Case 2: $k=3$ and $5\le n\le 8$.}  Then  \begin{equation}\label{2} |A_H(x_i,x_{i+1})|\ge \frac{1}{2}\binom{n-2}{k-2}\geq  \frac{1}{2}\binom{3}{1}=\frac{3}{2}\end{equation} for $i=1,2,3,4.$ Consider a bipartite graph $G$ with partite sets $Z=\{z_1,z_2,z_3,z_4\}$ and $A(H).$
 We have an edge $z_ia_j$ if $a_j\in A_H(x_i,x_{i+1}).$ By (\ref{2}), each vertex in $Z$ has degree at least two. Since $k=3$, vertices $z_i$ and $z_j$ in $G$ have no common neighbor unless $|i-j|=1.$ Thus, every vertex of $G$ in $A(H)$ has degree at most 2. Thus, by Lemma \ref{l1},
$G$ has a matching saturating $Z$. In other words, there are distinct $a_1,a_2,a_3,a_4\in A(H)$ such that $x_1a_1x_2a_2x_3a_3x_4a_4x_5$ is a path in $H.$

\vspace{1mm}

 \noindent{\bf Case 3: $k=4$ and $5\le n\le 6$.}  Consider the bipartite graph $G$ constructed as in the previous case. Using the computations analogous to those in (\ref{2}), we see that the minimum degree of a vertex in $Z$ is  at least $3$ when $n=6$ and at least $2$ when $n=5.$
 Since $k=4$, there is no common neighbor of all vertices in $Z.$ Thus, every vertex of $G$ in $A(H)$ has degree at most $3$.
 Now consider two subcases.

\vspace{1mm}

 \noindent{\bf Subcase 1: $n=6$.}
  Since every vertex of $G$ in $A(H)$ has degree at most $3$ and every vertex of $G$ in $Z$ has degree at least $3$, by Lemma \ref{l1}, $G$ has a matching saturating $Z$ and we are done as in Case 2.

 \vspace{1mm}

 \noindent{\bf Subcase 2: $n=5$.}  Recall that the minimum degree of a vertex in $Z$ is at least $2$. Suppose that there are two vertices of $G$ in $A(H)$ of degree $3$. This means that
 \begin{equation}\label{3} N_G(z_i)\cap N_G(z_{i+1})\cap N_G(z_{i+2}) \neq\emptyset\end{equation} for $i=1$ or $2$. Indeed, since $k=4$, $N_G(z_1)\cap N_G(z_{j})\cap N_G(z_{4})=\emptyset$
 when either $j=2$ or $3$. Without loss of generality, we assume that (\ref{3}) holds when $i=1$ and  let $e_1\in  N_G(z_1)\cap N_G(z_{2})\cap N_G(z_{3}).$ Thus, $e_1=x_1x_2x_3x_4.$

 If $x_1$ and $x_4$ are in different partite sets of $H$, then $x_1{e_1}{x_4}$. Since $e_1$ does not contain $x_5$,
 we can choose an arc $e_2$ of $H$ which is different from $e_1$ such that $x_4{e_2}{x_5}$. Then $x_1{e_1}{x_4}{e_2}{x_5}$ is a path in $H$.
 Now we assume that $x_1$ and $x_4$ are in the same partite set of $H.$ Then there is an arc $e_1$ of $H$ such that $x_1{e_1}{x_3}$.
Since the degree of $z_3$ in $G$ is at least $2$, we can choose an arc $e_2$ of $H$ which is different from $e_1$ such that $x_3{e_2}{x_4}.$ We can also choose an arc $e_3$ of $H$ which is different from $e_1$ and $e_2$ such that $x_4{e_3}{x_5}$. Indeed, $e_3\ne e_1$ since $e_1$ does not contain $x_5$
and $e_3\ne e_2$ since the degree of $z_4$ in $G$ is at least $2$.
Then $x_1{e_1}{x_3}{e_2}{x_4}{e_3}{x_5}$ is a path in $H$.  Thus, we may assume that every vertex of $G$ in $A(H)$ has degree at most $2$, except for one vertex which has degree at most $3$. Then we can use Lemma \ref{l1} and thus we are done as above.

\vspace{1mm}

 \noindent{\bf Case 4: $k\in \{5, 6,7\}$ and $n=k+1$.} Consider the bipartite graph $G$ constructed as in Case $2$.

 \vspace{1mm}

 \noindent{\bf Subcase 1: $k\in \{6,7\}$.} Using the computations analogous to those in (\ref{2}), we see that the minimum degree of a vertex in $Z$ is at least $3$. If there is a vertex with degree $4$ in $A(H)$, then it means $\{x_1,x_2,x_3,x_4,x_5\}$ is a subset of a vertex set of an arc $e_1$ and the relative order is $x_1,x_2,x_3,x_4,x_5$. If $x_1$ and $x_5$ are in different partite sets, then $x_1{e_1}x_5$ is a path in $H$. Otherwise $x_1$ and $x_4$ are in different partite sets, so $x_1{e_1}x_4.$  There is an arc $e_2$ different from $e_1$ such that $x_4{e_2}x_5$ (since the degree of $z_4$ is at least $3$). Now $x_1{e_1}x_4{e_2}x_5$ is a path in $H.$ Thus, we assume each vertex in $A(H)$ has degree at most $3$, and we are done by Lemma \ref{l1}.
 
  \vspace{1mm}

 \noindent{\bf Subcase 2: $k=5$.} Suppose that the lemma does not hold in this case. 
 Using the computations analogous to those in (\ref{2}), we see that the minimum degree of a vertex in $Z$ is at least $2$.
 To obtain a contradiction, it suffices to show that $G$ has at most one vertex of degree at least $3$ in $A(H).$
 Suppose that $G$ has at least two vertices of degree at least $3$ in $A(H).$ This means that (\ref{3}) holds for $i=1$ or $2$.
 Since $H$ can have only one arc with vertex set $\{x_1,x_2,x_3,x_4,x_5\},$ we have
\begin{equation} \sum^3_{j=2}|N_G(z_1)\cap N_G(z_{j})\cap N_G(z_{4})|\leq 1 \end{equation}
 Without loss of generality, we assume that (\ref{3}) holds when $i=1$ and  let $e_1\in  N_G(z_1)\cap N_G(z_{2})\cap N_G(z_{3}).$ If we restrict $e_1$ to the vertices $\{x_1,x_2,x_3,x_4\},$ we
obtain $e'_3=x_1x_2x_3x_4.$

If $x_1$ and $x_4$ are in the different partite sets, then $x_1{e_1}{x_4}$. Since the degree of $z_4$ in $G$ is at least 2, we can choose an arc $e_2$ of $H$ which is different from $e_1$ such that $x_4{e_2}{x_5}$. Then $x_1{e_1}{x_4}{e_2}{x_5}$ is a path in $H$, a contradiction. Now we assume $x_1$ and $x_4$ are in the same partite set. Then $x_1{e_1}{x_3}$. Since the degree of $z_3$ in $G$ is at least 2, we can choose an arc $e_2$ of $H$ which is different from $e_1$ such that $x_3{e_2}{x_4}.$ Since the degree of $z_4$ in $G$ is at least $2$, we can choose an arc $e_3$ of $H$ such that $x_4{e_3}{x_5}$ and $e_3\ne e_2.$
Suppose $e_3=e_1.$ Then $e_1=x_1x_2x_3x_4x_5$ and $x_1e_1x_5,$ a contradiction.  Thus, $e_3\ne e_1$ and $x_1{e_1}{x_3}{e_2}{x_4}{e_3}{x_5}$ is a path in $H$, a contradiction.  
\end{proof}

\section{Main Results}\label{sec:res}

In Section \ref{sec:res1}, using the Majority Lemma and other results, we solve Conjecture \ref{c1} in affirmative. In Section \ref{sec:res2}, we describe a family of couterexamples to Conjecture \ref{c2} and prove a sufficient condition of when the statement of Conjecture \ref{c2} holds.

\subsection{Results on Conjecture  \ref{c1}}\label{sec:res1}

\begin{lemma}\label{l3}
Let $H=(V,A)$ be a multipartite $k$-hypertournament with at most one transmitter and let $M(H)$ be a majority multipartite tournament of $H$.
Let $n\ge 5$ and $n>k\ge 3.$
If $M(H)$ has at least one transmitter, then $H$ has a 2-king.
\end{lemma}
\begin{proof}
Let $V_1$ be the partite vertex set containing all transmitters  of $M(H).$ Let $v$ be the transmitter of $H,$ if $H$ has a transmitter, and an arbitrary transmitter of $M(H)$, otherwise. Clearly, $v\in V_1.$
Observe that for every $u\in V\setminus V_1$, there is an arc $a\in A_H(v,u)$ implying that $vau.$ Note that for every $w\in V_1\setminus \{v\},$ there are a vertex $u\in V\setminus V_1$ and an arc $e$ of $H$ such that
$uew.$ As in Lemma \ref{l2}, it is easy to see that  $|A_H(v,u)|\ge 2.$ Thus, there is an arc $a\in A_H(v,u)$ distinct from $e$ implying that $vauew$ is a path.
\end{proof}

\begin{lemma}\label{l4}
Let $H=(V,A)$ be a multipartite $k$-hypertournament and let $n\ge 5$ and $n>k\ge 3.$ If $H$ has at most one transmitter then $H$ has a 4-king.
\end{lemma}
\begin{proof}
Let $M(H)$ be a majority multipartite tournament of $H$. If $M(H)$ has no transmitters, then by Theorem \ref{t4-k}, $M(H)$ has a 4-king $x$. By Lemma \ref{l2}, $x$ is a 4-king of $H.$
If $M(H)$ has transmitters, then we apply Lemma \ref{l3}.
\end{proof}

\begin{lemma}\label{l5}
Let $H=(V,A)$ be a $p$-partite $k$-hypertournament with $k=3,$  $n=4$ and $p\ge 2.$ If $H$ has at most one transmitter then $H$ has a 4-king.
\end{lemma}
\begin{proof}
By Theorem \ref{bt4-k}, this lemma holds for $p=2$ and so we may assume that $p\ge 3.$ It is well known that every $k$-hypertournament with more than $k$ vertices has a Hamilton path \cite{GY}. Observe that for $p=4$ the first vertex of a Hamilton path in $H$ is a 3-king.
Now we may assume that $p=3.$
Let  $V=V_1\cup V_2 \cup V_3$ be a partition of vertices of $H$. Without loss of generality, we may assume that $V_1=\{x_1,x_2\}$, $V_2=\{x_3\}$ and $V_3=\{x_4\}$.

First assume that $H$ has the unique transmitter $v.$ If $v=x_3$ or $v=x_4$, then $v$ is a 1-king of $H$. Thus, we assume without loss of generality that $v=x_1.$ Since $v$ is a transmitter, $va_1x_3$ and $va_2x_4$ for some arcs $a_1$ and $a_2$ of $H.$
Since $x_2$ is not a transmitter, there is an arc $e_1$ such that $y{e_1}x_2$, where $y\in V_2 \cup V_3$. By the definition of a transmitter, $v$ precedes $y$ in every arc containing $v$ and $y$. Consequently, there is an arc $e_2$ different from $e_1$ such that $v{e_2}y$.
Hence $v{e_2}y{e_1}{x_2}$ is a path from $v$ to $x_2.$ So $v$ is a 2-king.

Now assume that $T$ has no transmitter. Consider the arc $e_1$ containing $x_1$, $x_3$, and $x_4.$ If $x_1$ is in the first position of $e_1$, since $x_2$ is not a transmitter, there is an arc $e_2$ different from $e_1$ such that $x_3{e_2}x_2$ or $x_4{e_2}x_2$.
Hence $x_1{e_1}x_3{e_2}{x_2}$ or $x_1{e_1}x_4{e_2}{x_2}$ is a path from $x_1$ to $x_2$, implying that  $x_1$ is a 2-king. Without loss of generality, we now assume that $x_3$ is in the first position of $e_1$. Since $x_2$ is not a transmitter, there is an arc $e_2$,
where $x_3$ or $x_4$ preceds $x_2$.
Hence $x_3$ is a 2-king.
\end{proof}

Lemmas \ref{l4} and \ref{l5} imply the following result solving Conjecture \ref{c1} in affirmative.

\begin{thm}
Every multipartite hypertournament with at most one transmitter has a 4-king.
\end{thm}


\subsection{Results on Conjecture  \ref{c2}}\label{sec:res2}

The next result describes a family of counterexamples to Conjecture \ref{c2}.

\begin{prop}
For every $k\ge 3,$ there is a bipartite $k$-hypertournament $B$ without transmitters which has at most one 4-king in each of its partite sets.
\end{prop}
\begin{proof}
Let $U$ and $W$ be partite sets of $B.$
Choose a vertex $u$ in $U$  and a vertex $w$ in $W.$ Let every arc of $B$ with both $u$ and $w$ have both of them in the first and second position such that in at least one such arc $u$ is the first and in at least one such arc $w$ is the first. 
Let every arc of $B$ containing $u$ but not $w$ have $u$ in the first position and let every arc of $B$ containing $w$ but not $u$ have $w$ in the first position.
Clearly, $B$ has no transmitters, but no vertex $v$ in $(U\cup W)\setminus \{u,w\}$ can be a 4-king as there is no path from $v$ to either $u$ or $w.$
\end{proof}

The next result is a sufficient condition of when the conclusion of Conjecture \ref{c2} holds. It follows directly from Theorem \ref{thm:PT} and the Majority Lemma.

\vspace{1mm}

\begin{thm}\label{t1}
Let $B$ be a bipartite hypertournament with partite sets $U$ and $W$ and with at least 5 vertices. If a majority bipartite tournament $M(B)$ has no transmitters, then $B$ has at least two 4-kings in each $U$ and $W.$
\end{thm}

Our final result shows that the Majority Lemma cannot be extended to $n=4$ and $p=2.$ The proof provides another counterexample to Conjecture \ref{c2}.

 \begin{prop}\label{p2}
 For $k=3$ and $n=4,$ there is a bipartite hypertournament $H$ with partite sets $U$ and $W$ such that
(i) $|U|=|W|=2$, (ii) a majority bipartite tournament $M(H)$ has no transmitters, (iii) $M(H)$ has an $(x,y)$-path of length 3, but $H$ has no $(x,y)$-path, (iv) $H$ has only one 4-king in $U.$
\end{prop}
\begin{proof}
Let $H$ be a bipartite hypertournament with partite sets $U=\{x_1,x_3\}$ and $W=\{x_2,x_4\}$, arc set $\{a_1,a_2,a_3,a_4\}$ where $$a_1=x_{4}x_{1}x_{2}, a_2=x_{2}x_{3}x_{4}, a_3=x_{3}x_{2}x_{1}, a_4=x_{4}x_{3}x_{1}.$$ Let the arcs of $M(H)$
be $x_4x_1,x_1x_2,x_{2}x_{3},x_{3}x_{4}$  (see Fig. \ref{fig:M(H)}). Clearly, (i) and (ii) hold and
$x_{1}x_{2}x_{3}x_{4}$ is an $(x_{1},x_{4})$-path in $M(H)$.

\begin{figure}
\begin{center}
\begin{tikzpicture}
\fill (-1,1) circle (0.05);
\fill (1,1) circle (0.05);
\fill (-1,-1) circle (0.05);
\fill (1,-1) circle (0.05);
\draw [-latex] (-1,1)--(1,1);
\draw [-latex] (1,1)--(-1,-1);
\draw [-latex] (-1,-1)--(1,-1);
\draw [-latex] (1,-1)--(-1,1);
\node[left] at (-1,1)  {$x_1$};
\node[right] at (1,1)  {$x_2$};
\node[left] at (-1,-1)  {$x_3$};
\node[right] at (1,-1)  {$x_4$};
\end{tikzpicture}
\end{center}\caption{$M(H)$}\label{fig:M(H)}
\end{figure}

Now consider $H.$ Suppose that $H$ has an $(x_{1},x_{4})$-path $P.$ Since $A_B(x_1,x_4)=\emptyset$, $P=x_1b_1x_2b_2x_3b_3x_4$
for some distinct arcs  $b_1,b_2,b_3$ of $H.$  By inspection of the arcs of $H$, we conclude that $b_1=a_1, b_2=a_2, b_3=a_2,$ which is impossible since $b_1,b_2,b_3$ must be distinct. So $H$ has no $(x_{1},x_{4})$-path and (iii) holds.
Observe that $x_3$ is a 4-king of $H$ since $x_3a_3x_2,$ $x_3a_2x_4$ and $x_{3}a_{2}x_{4}a_{1}x_{1}$ is an $(x_{3,}x_{1})$-path of length 2. Moreover, $x_1$ cannot be a 4-king by the discussion in (iii), so (iv) holds.\end{proof}

\end{document}